\documentclass[
final
]{dmtcs-episciences}

\usepackage[utf8]{inputenc}
\usepackage{subfigure}
\usepackage{latexsym}
\usepackage{mathrsfs}
\usepackage{amsmath,amsthm}
\usepackage{graphicx}
\usepackage{amssymb}
\usepackage{CJK}
\usepackage{color}
\usepackage{epstopdf}

\newtheorem{thm}{Theorem}[section]

\newtheorem{obs}[thm]{Observation}
\newtheorem{lem}[thm]{Lemma}

\DeclareGraphicsExtensions{.eps,.eps.gz}
\normalsize \rm
\usepackage{geometry}
\usepackage[centerlast]{caption}

\author{Wei Zhuang\affiliationmark{1}\thanks{Corresponding author.}
\and Guoliang Hao\affiliationmark{2}}
\title{Semitotal domination in trees \thanks{This work is fully supported by NSFC (No. 11301440)
and Natural Science Foundation of Fujian Province
(CN)(2015J05017).}}

\affiliation{
School of Applied Mathematics, Xiamen University of Technology, P.R.China\\
College of Science, East China University of Technology, P.R.China}


\received{2018-3-29}
\revised{2018-7-19}
\accepted{2018-9-13}

\begin{document}
\publicationdetails{20}{2018}{2}{5}{4413}
 \maketitle
\begin{abstract}
In this paper, we study a parameter that is squeezed between
arguably the two important domination parameters, namely the
domination number, $\gamma(G)$, and the total domination number,
$\gamma_t(G)$. A set $S$ of vertices in $G$ is a semitotal
dominating set of $G$ if it is a dominating set of $G$ and every
vertex in S is within distance $2$ of another vertex of $S$. The
semitotal domination number, $\gamma_{t2}(G)$, is the minimum
cardinality of a semitotal dominating set of $G$. We observe that
$\gamma(G)\leq \gamma_{t2}(G)\leq \gamma_t(G)$. In this paper, we
give a lower bound for the semitotal domination number of trees and
we characterize the extremal trees. In addition, we characterize
trees with equal domination and semitotal domination numbers.
\end{abstract}

\section{Introduction}
\label{sec:in}

Let $G=(V, E)$ be a graph without isolated vertices with vertex set
$V$ of order $n(G)=|V|$ and edge set $E$ of size $m(G)=|E|$, and let
$v$ be a vertex in $V$. The \emph{open neighborhood} of $v$ is
$N(v)=\{u\in V|uv\in E\}$ and the \emph{closed neighborhood} of $v$
is $N[v]=N(v)\cup \{v\}$. The \emph{degree} of a vertex $v$ is
$d(v)=|N(v)|$. For two vertices $u$ and $v$ in a connected graph
$G$, the \emph{distance} $d(u, v)$ between $u$ and $v$ is the length
of a shortest $(u, v)$-path in $G$. The maximum distance among all
pairs of vertices of $G$ is the \emph{diameter} of a graph $G$ which
is denoted by $diam(G)$. A \emph{leaf} of $G$ is a vertex of degree
$1$ and a \emph{support vertex} of $G$ is a vertex adjacent to a
leaf. Denote the sets of leaves and support vertices of $G$ by
$L(T)$ and $S(T)$, respectively. Let $l(T)=|L(T)|$ and
$s(T)=|S(T)|$. A \emph{double star} is a tree that contains exactly
two vertices that are not leaves.

A dominating set in a graph $G$ is a set $S$ of vertices of $G$ such
that every vertex in $V(G)\setminus S$ is adjacent to at least one
vertex in $S$. The domination number of $G$, denoted by $\gamma(G)$,
is the minimum cardinality of a dominating set of $G$. A total
dominating set of a graph $G$ with no isolated vertex is a set $D$
of vertices of $G$ such that every vertex in $V(G)$ is adjacent to
at least one vertex in $D$. The total domination number of $G$,
denoted by $\gamma_t(G)$, is the minimum cardinality of a total
dominating set of $G$. A dominating (total dominating) set of $G$ of
cardinality $\gamma(G)$ ($\gamma_t(G)$) is called a $\gamma(G)$-set
($\gamma_t(G)$-set).

The concept of semitotal domination in graphs was introduced and
studied by Goddard, Henning and McPillan \cite{Goddard}. A set $S$
of vertices in a graph $G$ with no isolated vertices is a semitotal
dominating set of $G$ if it is a dominating set of $G$ and every
vertex in $S$ is within distance $2$ of another vertex of $S$. The
semitotal domination number, denoted by $\gamma_{t2}(G)$, is the
minimum cardinality of a semitotal dominating set of $G$. A
semitotal dominating set of $G$ of cardinality $\gamma_{t2}(G)$ is
called a $\gamma_{t2}(G)$-set. Clearly, for every graph $G$ with no
isolated vertex, $\gamma(G)\leq \gamma_{t2}(G)\leq \gamma_t(G)$. If
the graph $G$ is clear from the context, we simply write
$\gamma$-set and $\gamma_{t2}$-set rather than $\gamma(G)$-set and
$\gamma_{t2}(G)$-set, respectively.

An area of research in domination of graphs that has received
considerable attention is the study of classes of graphs with equal
domination parameters. For any two graph theoretic parameters
$\lambda$ and $\mu$, $G$ is called a $(\lambda, \mu)$-graph if
$\lambda(G)=\mu(G)$. The class of $(\gamma, \gamma_t)$-trees, that
is trees with equal domination and total domination numbers, was
characterized in \cite{Hou}. In \cite{Haynes}, the authors provided
a constructive characterizations of trees with equal domination and
paired domination numbers. More results in this area were
investigated in \cite{Krishnakumari, Li, Krzywkowski, Brause} and
elsewhere. Motivated by these results, we aim to characterize trees
with equal domination and semitotal domination numbers. In addition,
we give a lower bound for the semitotal domination number of trees
and we characterize the extremal trees.

\section{A lower bound for semitotal domination number of trees}
\label{sec:a low}

In this section we give a lower bound for the semitotal domination
number of trees and we characterize the extremal trees. First, we
shall need the following two observations.

\begin{obs}
Let $G$ be a connected graph that is not a star. Then,

$(i)$ there is a $\gamma$-set of $G$ that contains no leaf, and

$(ii)$\cite{Henning1} there is a $\gamma_{t2}$-set of $G$ that
contains no leaf.
\end{obs}

\begin{thm}
If $T$ is a tree of order $n(T)\geq 2$ with $l(T)$ leaves, then
$\gamma_{t2}(T)\geq \frac{2[n(T)-l(T)+2]}{5}$.
\end{thm}

\begin{proof}
We use induction on $n(T)$. It is easy to see that the result
holds for a tree of order $n\leq 8$. Let $T$ be a tree of order
$n>8$ and assume that $\gamma_{t2}(T')\geq
\frac{2[n(T')-l(T')+2]}{5}$ for each tree $T'$ with order at most
$n-1$. We consider the case that $diam(T)\geq 4$. Otherwise, $T$ is
a star or double-star, then $\gamma_{t2}(T)$ has the desired
property in theorem. By Observation~2.1(ii), we can obtain a
$\gamma_{t2}$-set of $T$, say $D$, which contains no leaf.

{\flushleft\textbf{Claim 1.}}\quad For any vertex $v\in
V(T)\setminus L(T)$, $v$ has only one leaf-neighbor when
$|N(v)\setminus L(T)|=1$, and $v$ is not a support vertex when
$|N(v)\setminus L(T)|\geq 2$.

\begin{proof}
If $v$ is a vertex that has at least two leaf-neighbors and
$|N(v)\setminus L(T)|=1$. We remove one of those leaves and denote
the resulting tree by $T'$. It is easy to observe that
$\gamma_{t2}(T')=\gamma_{t2}(T)$. By induction, $\gamma_{t2}(T')\geq
\frac{2[n(T')-l(T')+2]}{5}$. And consequently $\gamma_{t2}(T)\geq
\frac{2[n(T)-l(T)+2]}{5}$ as $l(T')=l(T)-1$, $n(T')=n(T)-1$.

If $v$ is a support vertex and $|N(v)\setminus L(T)|\geq 2$, we
remove a leaf-neighbor of $v$ and the semitotal domination number of
the resulting tree is no greater than that of $T$. Analogously to
the previous case, $\gamma_{t2}(T)$ has the desired property in
theorem. \end{proof}

In other words, each support vertex of $T$ has degree two. Let
$P=v_0v_1v_2 \cdots v_t$ be a longest path in $T$ such that

(i) $d(v_3)$ as large as possible, and subject to this condition

(ii) $d(v_2)$ as large as possible.

By Claim~1, $d(v_1)$=2 and $v_2$ is not a support vertex.

{\flushleft\textbf{Claim 2.}}\quad $d(v_2)=2$.

\begin{proof}
If $d(v_2)>2$, it follows from the choice of $P$ and Claim~1 that
all neighbors of $v_2$ are support vertices of degree two, except
possibly the vertex $v_3$.

Let $u_1$ be a neighbor of $v_2$ outside $P$, $u_2$ be the leaf that
adjacent to $u_1$, and $T'=T-\{v_0, u_2\}$. By induction,
$\gamma_{t2}(T')\geq \frac{2[n(T')-l(T')+2]}{5}$. In addition,
replacing the vertices $u_1$ and $v_1$ in $D$ with $v_2$(If $v_2\in
D$, take $D\setminus \{v_1\}$ instead), we can obtain a semitotal
dominating set of $T'$. That is, $\gamma_{t2}(T')\leq
\gamma_{t2}(T)-1$. Note that $l(T')=l(T)$, $n(T')=n(T)-2$.
Therefore, $\gamma_{t2}(T)\geq \frac{2[n(T)-l(T)+2]}{5}$.
\end{proof}

We know that $v_1\in D$ and exactly one of $v_2$ and $v_3$ belongs
to $D$. Without loss of generality, $v_3\in D$ (Otherwise, we
replace the vertex $v_2$ in $D$ with $v_3$, and the resulting set is
also a $\gamma_{t2}$-set of $T$).

{\flushleft\textbf{Claim 3.}}\quad $d(v_3)=2$.

\begin{proof}
By Claim~1 and the assumption that $n>8$, $v_3$ is not a support
vertex. If $d(v_3)>2$, it follows from the choice of $P$ and Claim~1
that $v_3$ has a neighbor of degree two outside $P$, say $v_2'$,
which is either a support vertex or adjacent to a support vertex
outside $P$, say $v_1'$.

In the former case, we have that $\{v_1, v_3, v_2'\}\subseteq D$.
And in the latter case, we have that $\{v_1, v_3, v_1'\}\subseteq
D$. Let $T'=T-\{v_0, v_1\}$. By induction, $\gamma_{t2}(T')\geq
\frac{2[n(T')-l(T')+2]}{5}$. In either case, we have that
$l(T')=l(T)$, $n(T')=n(T)-2$ and it is easy to see that
$\gamma_{t2}(T')\leq \gamma_{t2}(T)-1$. Therefore,
$\gamma_{t2}(T)\geq \frac{2[n(T)-l(T)+2]}{5}$.\end{proof}

{\flushleft\textbf{Claim 4.}}\quad $d(v_4)=2$.

\begin{proof}
By Claim~1 and the assumption that $n>8$, $v_4$ is not a support
vertex. If $d(v_4)>2$, from the choice of $P$ and Claim~1, we only
need to consider the case as follows: $v_4$ has a neighbor outside
$P$, say $v_3'$, which is adjacent to $t$ support vertices $u_1,
u_2, \cdots, u_t$, where $t\geq 2$. (In other cases, we always have
that $\gamma_{t2}(T')\leq \gamma_{t2}(T)-1$, $l(T')=l(T)$ and
$n(T')=n(T)-2$, where $T'=T-\{v_0, v_1\}$. And similar to the proof
of Claim~3, $\gamma_{t2}(T)$ has the desired property in theorem.)
Let $u_i'$ be the leaf-neighbor of $u_i$, where $i=1, 2, \cdots, t$.
Let $T'=T-\{u_1', u_2', \cdots, u_t'\}$. By induction,
$\gamma_{t2}(T')\geq \frac{2[n(T')-l(T')+2]}{5}$. Note that $\{u_1,
u_2, \cdots, u_t\}\subseteq D$. Then $(D\setminus \{u_1, u_2,
\cdots, u_t\})\cup \{v_3'\}$ is a semitotal dominating set of $T'$.
That is, $\gamma_{t2}(T')\leq \gamma_{t2}(T)-t+1$. In addition,
$l(T')=l(T)$, $n(T')=n(T)-t$. Hence, $\gamma_{t2}(T)\geq
\frac{2[n(T')-l(T')+2]}{5}+t-1=\frac{2[n(T)-t-l(T)+2]}{5}+t-1>\frac{2[n(T)-l(T)+2]}{5}$.
\end{proof}

Note $v_1, v_3\in D$. Then, one of the two cases as following holds:
(1) Each vertex of $D\setminus \{v_1, v_3\}$ is at distance at least
$3$ from $v_3$; (2) There is a vertex of $D\setminus \{v_1, v_3\}$
which is within distance $2$ of $v_3$.

In the former case, let $T'=T-\{v_0, v_1, v_2, v_3, v_4\}$. By
induction, $\gamma_{t2}(T')\geq \frac{2[n(T')-l(T')+2]}{5}$. In
addition, note that $D\setminus \{v_1, v_3\}$ is a semitotal
dominating set of $T'$, $n(T)=n(T')+5$, $l(T)\geq l(T')$. Hence,
$\gamma_{t2}(T)\geq \frac{2[n(T)-l(T)+2]}{5}$.

In the latter case, let $T'=T-\{v_0, v_1\}$. By induction,
$\gamma_{t2}(T')\geq \frac{2[n(T')-l(T')+2]}{5}$. Since $D\setminus
\{v_1\}$ is a semitotal dominating set of $T'$, $n(T)=n(T')+2$,
$l(T)=l(T')$. Hence, $\gamma_{t2}(T)\geq \frac{2[n(T)-l(T)+2]}{5}$.

The proof is completed. \end{proof}

Next, we are ready to provide a constructive characterization of the
trees achieving equality in the bound of Theorem~2.2. For our
purposes we define a \emph{labeling} of a tree $T$ as a partition
$S=(S_A, S_B, S_C)$ of $V(T)$ (This idea of labeling the vertices is
introduced in \cite{Dorfling}). We will refer to the pair $(T, S)$
as a \emph{labeled tree}. The label or \emph{status} of a vertex
$v$, denoted sta$(v)$, is the letter $x\in \{A, B, C\}$ such that
$v\in S_x$.

Let $\mathscr{T}$ be the family of labeled trees that: (i) contains
$(P_5, S')$ where $S'$ is the labeling that assigns to the two
support vertices of the path $P_5$ status $A$, to the two leaves
status $C$ and to the center vertex status $B$ (see Fig.1(a)); and
(ii) is closed under the two operations $\mathscr{O}_1$ and
$\mathscr{O}_2$ that are listed below, which extend the tree $T'$ to
a tree $T$ by attaching a tree to the vertex $v\in V(T')$.

{\bf Operation} $\mathscr{O}_1$: Let $v$ be a vertex with sta$(v)=A
$. Add a vertex $u$ and the edge $uv$. Let sta$(u)=C$.

{\bf Operation} $\mathscr{O}_2$: Let $v$ be a vertex with sta$(v)=C$
that has degree one. Add a path $u_1u_2u_3u_4u_5$ and the edge
$u_1v$. Let sta$(u_1)=$sta$(u_5)=C$, sta$(u_2)=$sta$(u_4)=A$,
sta$(u_3)=B$.

The two operations $\mathscr{O}_1$ and $\mathscr{O}_2$ are
illustrated in Fig.1(b), (c).

\begin{figure}[htbp!]
\centering
\includegraphics[height=6cm]{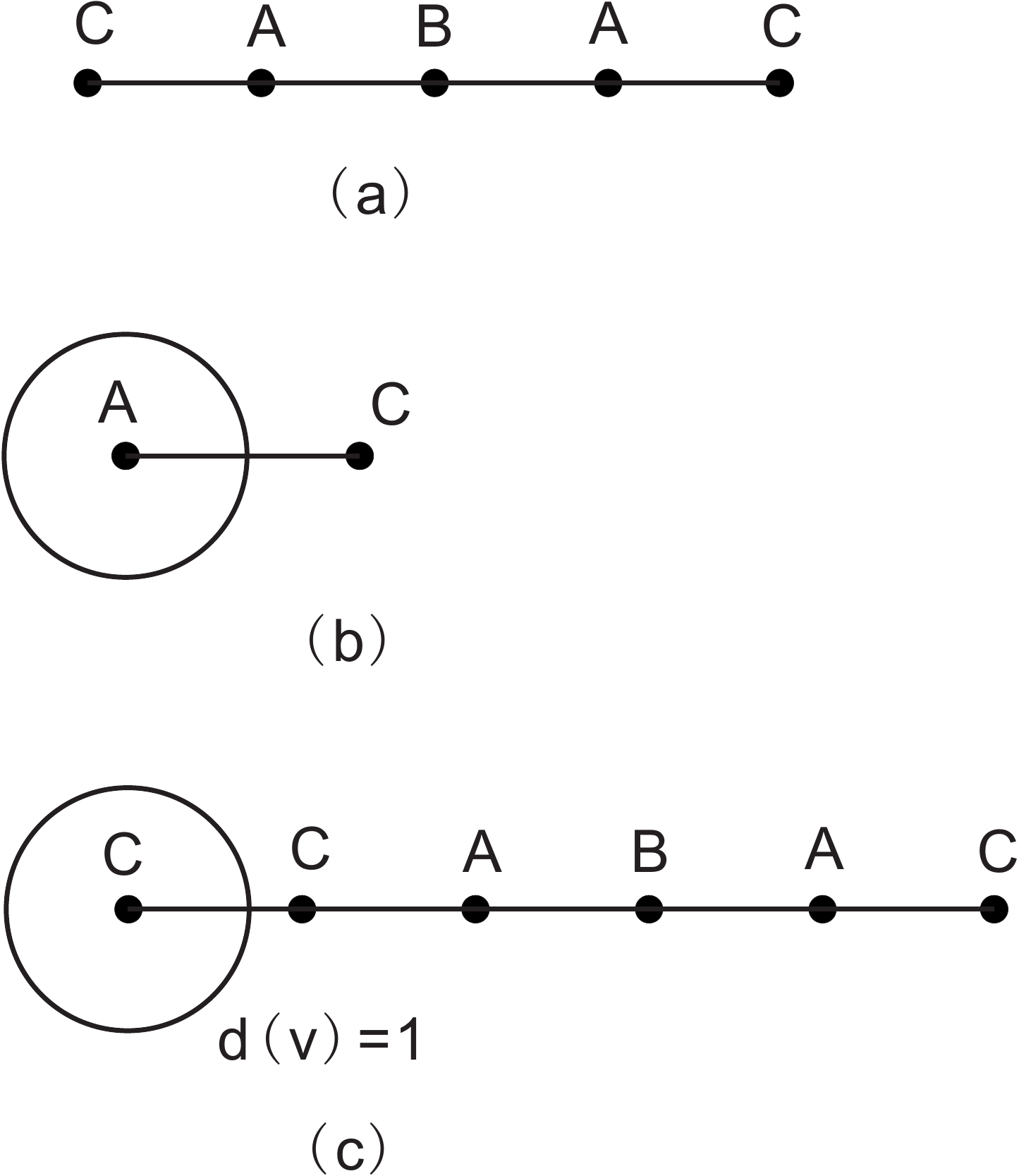}
\caption{\label{Fig.1}}
\end{figure}

Let $(T, S)\in \mathscr{T}$ be a labeled tree for some labeling $S$.
Then there is a sequence of labeled trees $(T_0, S_0)$, $(T_1, S_1),
\cdots, (T_{k-1}, S_{k-1})$, $(T_k, S_k)$ such that $(T_0,
S_0)=(P_5, S')$, $(T_k, S_k)=(T, S)$. The labeled tree $(T_i, S_i)$
can be obtained from $(T_{i-1}, S_{i-1})$ by one of the operations
$\mathscr{O}_1$ and $\mathscr{O}_2$, where $i\in \{1, 2, \cdots,
k\}$. We call the number of terms in such a sequence of labeled
trees that is used to construct $(T, S)$, the \emph{length} of the
sequence. Clearly, the above sequence has length $k$. We remark that
a sequence of labeled trees used to construct $(T, S)$ is not
necessarily unique.

We take an example to make it easier for reader to understand the
family $\mathscr{T}$. In Fig.2, $(P_5, S')\in \mathscr{T}$, $(H_1,
S_1)$ is obtained from $(P_5, S')$ by operation $\mathscr{O}_2$,
$(H_2, S_2)$ is obtained from $(H_1, S_1)$ by repeated applications
of operation $\mathscr{O}_1$, and $(H_3, S)$ is obtained from $(H_2,
S_2)$ by operation $\mathscr{O}_2$. Thus, $(H_1, S_1)$, $(H_2,
S_2)$, $(H_3, S)\in \mathscr{T}$. For $T\in \{P_5, H_1, H_2, H_3\}$,
it is easy to see that the set, say $D$, consisting of the vertices
labeled $A$ in $T$ is a $\gamma_{t2}$-set of $T$. In particular,
$|D|=\frac{2[n(T)-l(T)+2]}{5}$.

Before presenting our main result, we present a few preliminary
results and observations.

\begin{figure}[htbp!]
\centering
\includegraphics[height=7cm]{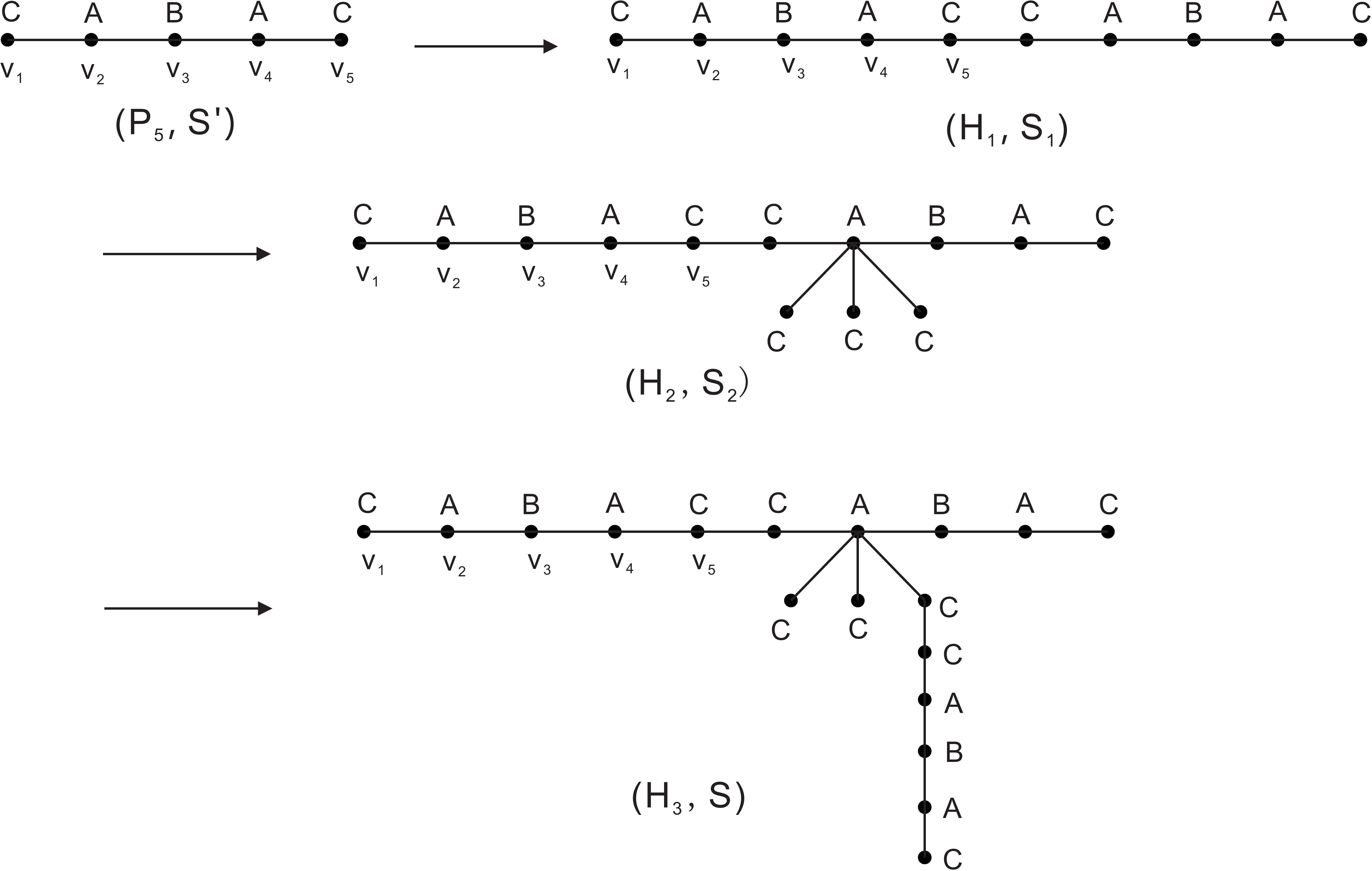}
\caption{\label{Fig.2}}
\end{figure}

\begin{obs}
Let $T$ be a tree and let $S$ be a labeling of $T$ such that $(T,
S)\in \mathscr{T}$. Then, $T$ has the following properties:

$(a)$ Every support vertex is labeled $A$ and every leaf is labeled
$C$.

$(b)$ $|S_A|=2|S_B|$.

$(c)$ The set $S_A$ is a semitotal dominating set of $T$.

$(d)$ The set $S_A$ and $S_B$ are independent sets.

$(e)$ Every vertex labeled $B$ has degree two and its neighbors
labeled $A$.
\end{obs}

\begin{lem}
Let $T$ be a tree and let $S$ be a labeling of $T$ such that $(T,
S)\in \mathscr{T}$. Then, $\gamma_{t2}(T)=
\frac{2[n(T)-l(T)+2]}{5}$.
\end{lem}

\begin{proof}
First, we are ready to show that $|S_A|=\frac{2[n(T)-l(T)+2]}{5}$.
We proceed by induction on the length $k$ of a sequence required to
construct the labeled tree $(T, S)$.

When $k=0$, $(T, S)=(P_5, S')$, and so $|S_A|=2$. This establishes
the base case. Let $k\geq 1$ and assume that if the length of
sequence used to construct a labeled tree $(T^{*}, S^{*})\in
\mathscr{T}$ is less than $k$, then
$|S^{*}_A|=\frac{2[n(T^{*})-l(T^{*})+2]}{5}$. Now, $(T, S)\in
\mathscr{T}$ and let $(T_0, S_0)$, $(T_1, S_1), \cdots, (T_{k-1},
S_{k-1})$, $(T_k, S_k)$ be a sequence of length $k$ used to
construct $(T, S)$, where $(T_0, S_0)=(P_5, S')$, $(T_k, S_k)=(T,
S)$, $(T_i, S_i)$ can be obtained from $(T_{i-1}, S_{i-1})$ by one
of the operations $\mathscr{O}_1$ and $\mathscr{O}_2$, $i\in \{1, 2,
\cdots, k\}$.  Let $T^{*}=T_{k-1}$ and $S^{*}=S_{k-1}$. Note that
$(T_{k-1}, S_{k-1})\in \mathscr{T}$. By the inductive hypothesis,
$|S^{*}_A|=\frac{2[n(T^{*})-l(T^{*})+2]}{5}$. $(T, S)$ can be
obtained from $(T^{*}, S^{*})$ by operation $\mathscr{O}_1$ or
$\mathscr{O}_2$.

In the former case, we have that $n(T)=n(T^{*})+1$,
$l(T)=l(T^{*})+1$, and $|S_A|=|S^{*}_A|$. Thus,
$|S_A|=\frac{2[n(T^{*})-l(T^{*})+2]}{5}=\frac{2[n(T)-1-l(T)+1+2]}{5}=\frac{2[n(T)-l(T)+2]}{5}$.

In the latter case, we have that $n(T)=n(T^{*})+5$, $l(T)=l(T^{*})$
and $|S_A|=|S^{*}_A|+2$. Thus,
$|S_A|=\frac{2[n(T^{*})-l(T^{*})+2]}{5}+2=\frac{2[n(T)-5-l(T)+2]}{5}+2=\frac{2[n(T)-l(T)+2]}{5}$.

By Observation~2.3(c), we have that $\gamma_{t2}(T)\leq
\frac{2[n(T)-l(T)+2]}{5}$. Combining Theorem~2.2, we conclude that
$\gamma_{t2}(T)= \frac{2[n(T)-l(T)+2]}{5}$. Moreover, $S_A$ is a
$\gamma_{t2}$-set of $T$. \end{proof}

\begin{thm}
Let $T$ be a nontrivial tree, then
$\gamma_{t2}(T)=\frac{2[n(T)-l(T)+2]}{5}$ if and only if $(T, S)\in
\mathscr{T}$ for some labeling $S$.
\end{thm}

\begin{proof}
 The sufficiency follows immediately from Lemma~2.4. So we prove
the necessity only. The proof is by induction on the order of $T$.
The result is immediate for $n\leq 5$. For the inductive hypothesis,
let $n\geq 6$ and moreover, $diam(T)\geq 4$ (If $diam(T)\leq 3$, $T$
is a star or a double star, and then
$\gamma_{t2}(T)>\frac{2[n(T)-l(T)+2]}{5}$, a contradiction). Assume
that for every nontrivial tree $T'$ of order less than $n$ with
$\gamma_{t2}(T')=\frac{2[n(T')-l(T')+2]}{5}$, we have that $(T',
S^{*})\in \mathscr{T}$ for some labeling $S^{*}$. Let $T$ be a tree
of order $n$ satisfying $\gamma_{t2}(T)=\frac{2[n(T)-l(T)+2]}{5}$.
Let $P=v_1v_2\cdots v_t$ be a longest path in $T$ such that

(i) $d(v_4)$ as large as possible, and subject to this condition

(ii) $d(v_3)$ as large as possible.

Let $D$ be a $\gamma_{t2}$-set of $T$ which contains no leaf.

{\flushleft\textbf{Claim 1.}}\quad Each support vertex has exactly
one leaf-neighbor.

\begin{proof}
If not, assume that there is a support vertex $u$ which is adjacent
to at least two leaves. Deleting one of its leaf-neighbors, say
$u_1$, and denote the resulting tree by $T'$. $D$ is still a
semitotal dominating set of $T'$. That is, $\gamma_{t2}(T')\leq
\gamma_{t2}(T)=\frac{2[n(T)-l(T)+2]}{5}=\frac{2[n(T')+1-l(T')-1+2]}{5}=\frac{2[n(T')-l(T')+2]}{5}$.
Combining Theorem~2.2, we have that $\gamma_{t2}(T')=
\frac{2[n(T')-l(T')+2]}{5}$. By the inductive hypothesis, $(T',
S^{*})\in \mathscr{T}$ for some labeling $S^{*}$. Since $u$ is still
a support vertex in $T'$, by Observation~2.3(a), the vertex $u$ has
label $A$ in $S^{*}$. Let $S$ be obtained from the labeling $S^{*}$
by labeling the vertex $u_1$ with label $C$. Then, $(T, S)$ can be
obtained from $(T', S^{*})$ by operation $\mathscr{O}_1$. Thus, $(T,
S)\in \mathscr{T}$. \end{proof}

By Claim~1, we can assume that $d(v_2)=2$. Now, we consider the
vertex $v_3$. If $v_3$ is a support vertex, then $v_2, v_3\in D$.
Let $T'$ be the tree which is obtained from $T$ by subdividing the
edge $v_2v_3$. It is easy to see that $D$ is still a semitotal
dominating set of $T'$, and it means that
$\frac{2[n(T)-l(T)+2]}{5}=\gamma_{t2}(T)\geq \gamma_{t2}(T')
 \geq
 \frac{2[n(T')-l(T')+2]}{5}=\frac{2[n(T)+1-l(T)+2]}{5}=\frac{2[n(T)-l(T)+3]}{5}$,
a contradiction. So, $v_3$ is not a support vertex.

Assume that $d(v_3)\geq 3$. Then, it follows from the choice of $P$
that $v_3$ is adjacent to a support vertex, say $u$, which does not
belong to $P$. Clearly, $u, v_2\in D$. Moreover, $v_3\not \in D$.
Otherwise, we subdivide the edges $v_2v_3$ and $uv_3$, and yield a
similar contradiction as above.

If $u$ is within distance two from a vertex in $D\setminus \{u,
v_2\}$, we have that $\frac{2[n(T)-2-(l(T)-1)+2]}{5}$ $\leq
\gamma_{t2}(T')\leq \gamma_{t2}(T)-1=\frac{2[n(T)-l(T)+2]}{5}-1$,
where $T'=T-\{v_1, v_2\}$. It is impossible. It follows that
$v_4\not \in D$, but in this case, let $T''$ be the component of
$T-v_3v_4$ containing the vertex $v_4$, and
$\frac{2[n(T)-l(T)+2]}{5}=\gamma_{t2}(T)\geq
\gamma_{t2}(T'')+\gamma_{t2}(T-T'')\geq
\frac{2[n(T'')-l(T'')+2]}{5}+\frac{2[n(T-T'')-l(T-T'')+2]}{5}\geq
\frac{2[n(T)-(l(T)+1)+4]}{5}=\frac{2[n(T)-l(T)+3]}{5}$, a
contradiction. Therefore, $d(v_3)=2$.

From the choice of $D$, $v_2\in D$, and without loss of generality,
$v_4\in D$ (If $v_4\not \in D$, then $v_3\in D$, replacing $v_3$ in
$D$ with $v_4$, and we obtain a new $\gamma_{t2}$-set of $T$).

Assume that $d(v_4)\geq 3$. We have that the following conclusion.

{\flushleft\textbf{Claim 2.}}\quad $N(v_4)\setminus \{v_3,
v_5\}\subset L(T)$.

\begin{proof}
Assume that there exists a vertex $v_3'\in N(v_4)\setminus \{v_3,
v_5\}$ which is not a leaf, it follows from the choice of $P$ and
Claim~1 that $v_3'$ is either a support vertex or adjacent to a
support vertex outside $P$, say $v_2'$. In particular, $d(v_3')=2$
(From the choice of $P$). In either case, let $T'=T-\{v_1, v_2\}$.
Observe that $n(T)=n(T')+2$, $l(T)=l(T')$, $\gamma_{t2}(T')\leq
\gamma_{t2}(T)-1$. Then, we have that $\gamma_{t2}(T')\leq
\gamma_{t2}(T)-1=\frac{2[n(T)-l(T)+2]}{5}-1=\frac{2[n(T')+2-l(T')+2]}{5}-1=\frac{2[n(T')-l(T')+2]}{5}-\frac{1}{5}<\frac{2[n(T')-l(T')+2]}{5}$,
contradicting Theorem~2.2. It concludes that $\emptyset \neq
N(v_4)\setminus \{v_3, v_5\}\subset L(T)$.\end{proof}

So $d(v_4)=2$ or $N(v_4)\setminus \{v_3, v_5\}\subset L(T)$.
Moreover, all vertices in $D\setminus \{v_2, v_4\}$ are distance at
least three from $v_4$ (If not, let $T'=T-\{v_1, v_2\}$. Observe
that $n(T)=n(T')+2$, $l(T)=l(T')$, $\gamma_{t2}(T')\leq
\gamma_{t2}(T)-1$. We can obtain a contradiction by an argument
similar to the proof of Claim~2).

If $d(v_5)=1$, by Claim~1 and the choice of $P$, $T=P_5$,
contradicting the assumption that $n\geq 6$. So assume that
$d(v_5)\geq 3$, since all vertices in $D\setminus \{v_2, v_4\}$ are
distance at least three from $v_4$, each neighbor of $v_5$ is
neither a leaf nor a support vertex. From the choice of $P$ and
Claim~1, we only need to consider the case as follows: $v_5$ has a
neighbor outside $P$, say $v_4'$, which is adjacent to $t$ support
vertices $u_1, u_2, \cdots, u_t$, where $t\geq 2$. (In other cases,
let $T'=T-\{v_1, v_2\}$. Observe that $n(T)=n(T')+2$, $l(T)=l(T')$,
$\gamma_{t2}(T')\leq \gamma_{t2}(T)-1$. We can always obtain
contradictions by an argument similar to the proof of Claim~2). Let
$u_i'$ be the leaf-neighbor of $u_i$, where $i=1, 2, \cdots, t$, and
$T'=T-\{u_1', u_2', \cdots, u_t'\}$. Note that $\{u_1, u_2, \cdots,
u_t\}\subseteq D$. Then $(D\setminus \{u_1, u_2, \cdots, u_t\})\cup
\{v_4'\}$ is a semitotal dominating set of $T'$. That is,
$\gamma_{t2}(T')\leq \gamma_{t2}(T)-t+1$. In addition, $l(T')=l(T)$,
$n(T')=n(T)-t$. Hence, $\gamma_{t2}(T')\leq
\gamma_{t2}(T)-t+1=\frac{2[n(T)-l(T)+2]}{5}-t+1=\frac{2[n(T')+t-l(T')+2]}{5}-t+1=
\frac{2[n(T')-l(T')+2]}{5}+\frac{2t}{5}-t+1<\frac{2[n(T')-l(T')+2]}{5}$,
contradicting Theorem 2.2. Therefore, $d(v_5)=2$.

Let $T'$ be the component of $T-v_5v_6$ containing $v_6$. If $v_6$
is not a leaf in $T'$, then $n(T)=n(T')+5+s$, $l(T)=l(T')+1+s$,
where $s$ is the number of the leaf-neighbors of $v_4$. Since all
vertices in $D\setminus \{v_2, v_4\}$ are distance at least three
from $v_4$, $\gamma_{t2}(T')\leq \gamma_{t2}(T)-2$. It follows that
$\gamma_{t2}(T')\leq
\frac{2[n(T)-l(T)+2]}{5}-2=\frac{2[n(T')+5+s-l(T')-1-s+2]}{5}-2=\frac{2[n(T')-l(T')+2]}{5}-\frac{2}{5}<\frac{2[n(T')-l(T')+2]}{5}$,
contradicting Theorem 2.2. It means that $v_6$ is a leaf in $T'$,
and $\gamma_{t2}(T')\leq \gamma_{t2}(T)-2
=\frac{2[n(T)-l(T)+2]}{5}-2=\frac{2[n(T')+5+s-l(T')-s+2]}{5}-2=\frac{2[n(T')-l(T')+2]}{5}$.
Combining Theorem~2.2, we have that $\gamma_{t2}(T')=
\frac{2[n(T')-l(T')+2]}{5}$. By the inductive hypothesis, $(T',
S^{*})\in \mathscr{T}$ for some labeling $S^{*}$. Since $v_6$ is a
leaf in $T'$, by Observation~2.3(a), the vertex $v_6$ has label $C$
in $S^{*}$.

If $d(v_4)=2$, let $S$ be obtained from the labeling $S^{*}$ by
labeling the vertices $v_1$ and $v_5$ with label $C$, the vertices
$v_2$ and $v_4$ with label $A$, the vertex $v_3$ with label $B$.
Then, $(T, S)$ can be obtained from $(T', S^{*})$ by operation
$\mathscr{O}_2$. Thus, $(T, S)\in \mathscr{T}$.

If $\emptyset \neq N(v_4)\setminus \{v_3, v_5\}\subset L(T)$, by
Claim~1, $v_4$ has exactly one leaf-neighbor. Let $S^{*}_1$ be
obtained from the labeling $S^{*}$ by labeling the vertices $v_1$
and $v_5$ with label $C$, the vertices $v_2$ and $v_4$ with label
$A$, the vertex $v_3$ with label $B$. $S$ be obtained from the
labeling $S^{*}_1$ by labeling the leaf-neighbor of $v_4$ with label
$C$. Then, $(T'', S^{*}_1)$ can be obtained from $(T', S^{*})$ by
operation $\mathscr{O}_2$, and $(T, S)$ can be obtained from $(T'',
S^{*}_1)$ by operation $\mathscr{O}_1$, where $T''$ is obtained from
$T$ by deleting the leaf-neighbor of $v_4$. Thus, $(T, S)\in
\mathscr{T}$.
\end{proof}

\section{A characterization of ($\gamma, \gamma_{t2}$)-trees}
\label{sec:a char}

Before presenting a characterization of ($\gamma,
\gamma_{t2}$)-trees, we shall need some additional notation.

Take a star with the center vertex $x$. A subdivided star, denoted
by $X$, is obtained from the star by subdividing all edges once. And
the tree obtained from the star by subdividing exactly one of the
edges once is denoted by $Y$.

An \emph{almost dominating set} (ADS) of $G$ relative to a vertex
$v$ is a set of vertices of $G$ that dominates all vertices of $G$,
except possibly for $v$. The \emph{almost domination number} of $G$
relative to $v$, denoted $\gamma(G; v)$, is the minimum cardinality
of an ADS of $G$ relative to $v$. An ADS of $G$ relative to $v$ of
cardinality $\gamma(G; v)$ we call a $\gamma(G; v)$-set.

In order to state the characterization of trees with equal
domination and semitotal domination numbers, we introduce the four
types of operations as follows.

{\bf Operation} $\mathscr{O}_1$: Add a path $P_1$ and join it to a
vertex of $T$, which is in some $\gamma_{t2}$-set of $T$.

{\bf Operation} $\mathscr{O}_2$: Add a path $P_2$ or $P_5$ and join
one of its leaves to a vertex $v$ of $T$, where $\gamma(T;
v)=\gamma(T)$.

{\bf Operation $\mathscr{O}_3$}: Add a subdivided star $X$ with at
least two leaves and join the center vertex $x$ to a vertex of $T$.

{\bf Operation $\mathscr{O}_4$}: Add $Y$ with three leaves and join
a leaf-neighbor of the center vertex $x$ to a vertex of $T$.

We define the family $\mathscr{O}$ as:

$\mathscr{O}=\{T|T$ is obtained from $P_4$ by a finite sequence of
operations $\mathscr{O}_i$, $i=1, 2, 3, 4\}$. We show first that
every tree in the family $\mathscr{O}$ has equal domination and
semitotal domination numbers.

\begin{lem}
If $T\in \mathscr{O}$, then $T$ is a $(\gamma, \gamma_{t2})$-tree.
\end{lem}

\begin{proof} The proof is by induction on the number $h(T)$ of operations
required to construct the tree $T$. Observe that $T=P_4$ when
$h(T)=0$, and clearly $\gamma(T)=\gamma_{t2}(T)$. This establishes
the base case. Assume that $k\geq 1$ and each tree $T'\in
\mathscr{O}$ with $h(T')<k$ is a $(\gamma, \gamma_{t2})$-tree. Let
$T\in \mathscr{O}$ be a tree with $h(T)=k$. Then $T$ can be obtained
from a tree $T'\in \mathscr{O}$ with $h(T')<k$ by one of the
operations $\mathscr{O}_i$, $i=1, 2, 3, 4$. By induction, $T'$ is a
$(\gamma, \gamma_{t2})$-tree. By Observation~2.1(i), we can obtain a
$\gamma$-set of $T$, say $S$, which contains no leaf. Now we can
distinguish four cases as follows:

{\flushleft\textbf{Case 1.}}\quad $T$ is obtained from $T'$ by
operation $\mathscr{O}_1$.

In this case, $T$ is obtained from $T'$ by adding a path $P_1$ and
joining it to a vertex of $T'$, which is in some $\gamma_{t2}$-set
of $T'$, say $D'$. Note that $D'$ is also a semitotal dominating set
of $T$. That is, $\gamma_{t2}(T')\geq \gamma_{t2}(T)$. On the other
hand, we have that $\gamma(T')=\gamma_{t2}(T')$. Moreover, since the
set $S$ contains no leaf of $T$, we have that $S$ is a dominating
set of $T'$, and then $\gamma(T')\leq \gamma(T)$. Hence,
$\gamma(T)\leq \gamma_{t2}(T)\leq \gamma_{t2}(T')=\gamma(T')\leq
\gamma(T)$. Consequently we must have equality throughout this
inequality chain. In particular, $\gamma(T)=\gamma_{t2}(T)$.

{\flushleft\textbf{Case 2.}}\quad $T$ is obtained from $T'$ by
operation $\mathscr{O}_2$.

First, suppose that $T$ is obtained from $T'$ by adding a path $P_2$
and joining one of its vertices, say $u$, to a vertex $v$ of $T'$,
where $\gamma(T'; v)=\gamma(T')$. Let $D'$ be a $\gamma_{t2}$-set of
$T'$. Clearly, $D'\cup \{u\}$ is a semitotal dominating set of $T$.
That is, $\gamma_{t2}(T)\leq \gamma_{t2}(T')+1$. On the other hand,
because $u\in S$, the set $S\setminus \{u\}$ can dominate all
vertices of $T'$, except possibly the vertex $v$. It follows from
the condition $\gamma(T'; v)=\gamma(T')$ that $\gamma(T)-1\geq
\gamma(T')$. Therefore, $\gamma(T)\leq \gamma_{t2}(T)\leq
\gamma_{t2}(T')+1=\gamma(T')+1\leq \gamma(T)$. It means that
$\gamma(T)=\gamma_{t2}(T)$.

Next, suppose that $T$ is obtained from $T'$ by adding a path $P_5$
and joining one of its leaves to a vertex $v$ of $T$, where
$\gamma(T; v)=\gamma(T)$. Analogously to the previous arguments, we
can deduce that $\gamma(T)=\gamma_{t2}(T)$.

{\flushleft\textbf{Case 3.}}\quad $T$ is obtained from $T'$ by
operation $\mathscr{O}_3$.

In this case, $T$ is obtained from $T'$ by adding a subdivided star
$X$ with at least two leaves and joining the center vertex $x$ to a
vertex of $T'$. The set $D_1$ consists of a $\gamma_{t2}$-set of
$T'$ together with all support vertices of $X$. Clearly, $D_1$ is a
semitotal dominating set of $T$. Assume that $X$ contains $t$ leaves
($t\geq 2$). Then, $\gamma_{t2}(T)\leq \gamma_{t2}(T')+t$. Moreover,
it is easy to see that $\gamma(T)-t\geq \gamma(T')$. So,
$\gamma(T)\leq \gamma_{t2}(T)\leq \gamma_{t2}(T')+t=\gamma(T')+t\leq
\gamma(T)$. Consequently we must have equality throughout this
inequality chain. In particular, $\gamma(T)=\gamma_{t2}(T)$.

{\flushleft\textbf{Case 4.}}\quad $T$ is obtained from $T'$ by
operation $\mathscr{O}_4$.

In this case, we can prove $\gamma(T)=\gamma_{t2}(T)$ similar to the
proof of Case~3. \end{proof}

\begin{lem}
If $T$ is a $(\gamma, \gamma_{t2})$-tree, then $T\in \mathscr{O}$.
\end{lem}

\begin{proof} We only need to consider the case that $n(T)\geq 6$ and
$diam(T)\geq 4$. Otherwise, $T=P_4$ or $T$ can be obtained from
$P_4$ by repeated applications of operation $\mathscr{O}_1$. We
proceed by induction on the order $n(T)$ of a $(\gamma,
\gamma_{t2})$-tree $T$. Assume that the result is true for all
$(\gamma, \gamma_{t2})$-tree $T'$ of order $n(T')<n(T)$. By
Observation~2.1(ii), we can obtain a $\gamma_{t2}$-set of $T$, say
$D$, which contains no leaf. Let $P=v_0v_1v_2\cdots v_s$ be a
longest path of $T$ such that

(i) $d(v_3)$ as large as possible, and subject to this condition

(ii) $d(v_2)$ as large as possible.

Let $z$ be a support vertex of $T$ which has at least two
leaf-neighbors. We remove one of these leaves and denote the
resulting tree by $T'$. Note that $D$ is still a semitotal
dominating set of $T'$. That is, $\gamma_{t2}(T')\leq
\gamma_{t2}(T)$. By Observation~2.1(i), there is a $\gamma$-set of
$T'$, say $S'$, which contains no leaf. Clearly, $z\in S'$ and then
$S'$ is also a dominating set of $T$. Therefore, $\gamma(T')\leq
\gamma_{t2}(T')\leq \gamma_{t2}(T)=\gamma(T)\leq \gamma(T')$.
Consequently we must have equality throughout this inequality chain.
In particular,  $\gamma(T')= \gamma_{t2}(T')$ and $z$ is in a
$\gamma_{t2}(T')$-set. By induction, $T'\in \mathscr{O}$. And then,
$T$ is obtained from $T'$ by operation $\mathscr{O}_1$. So, we
assume that each support vertex of $T$ is adjacent to exactly one
leaf, for otherwise, we are done. For this reason, $d(v_1)=2$.

We can distinguish two cases as follows.

{\flushleft\textbf{Case 1.}}\quad $v_2$ is a support vertex of $T$.

In this case, $v_1, v_2\in D$. Because of $diam(T)\geq 4$, $|D|\geq
3$. And then, one of the two cases as following holds: (1) Each
vertex of $D\setminus \{v_1, v_2\}$ is at distance at least $3$ from
$v_2$; (2) There is a vertex of $D\setminus \{v_1, v_2\}$ which is
within distance $2$ of $v_2$.

In the former case, if $d(v_3)\geq 3$, let $v_2'$ be a neighbor of
$v_3$ outside $P$. From the choice of $P$ and $D$, it is not
difficult to verify that the component of $T-v_2'v_3$ containing the
vertex $v_2'$ is a subdivided star with at least two leaves, say
$X$. Suppose that $X$ contains $t$ leaves. The set obtained by
deleting all support vertices of $X$ from $D$ is denoted by $D'$, is
still a semitotal dominating set of $T-X$. So, $\gamma_{t2}(T-X)\leq
\gamma_{t2}(T)-t$. On the other hand, the set consists of a
$\gamma$-set of $T-X$ together with all support vertices of $X$ is a
dominating set of $T$. For this reason, $\gamma(T)\leq
\gamma(T-X)+t$. Therefore, $\gamma(T-X)\leq \gamma_{t2}(T-X)\leq
\gamma_{t2}(T)-t=\gamma(T)-t\leq \gamma(T-X)$. It concludes that
$\gamma(T-X)= \gamma_{t2}(T-X)$. By induction, $T-X\in \mathscr{O}$.
Then, $T$ is obtained from $T-X$ by operation $\mathscr{O}_3$. If
$d(v_3)=2$, then the component of $T-v_3v_4$ containing $v_3$ is a
tree $Y$ with three leaves. With a similar discussion as above, one
can prove that $T$ is obtained from $T-Y$ by operation
$\mathscr{O}_4$.

In the latter case, let $T'=T-\{v_0, v_1\}$. Clearly, the inequality
chain $\gamma(T')\leq \gamma_{t2}(T')\leq
\gamma_{t2}(T)-1=\gamma(T)-1\leq \gamma(T')$ holds. And then,
$\gamma(T')= \gamma_{t2}(T')$. By induction, $T'\in \mathscr{O}$.
Further, we have that $\gamma(T)=\gamma(T')+1\geq \gamma(T';
v_2)+1\geq \gamma(T)$. That is, $\gamma(T')=\gamma(T'; v_2)$. Hence,
$T$ is obtained from $T'$ by operation $\mathscr{O}_2$.

{\flushleft\textbf{Case 2.}}\quad $v_2$ is not a support vertex of
$T$.

In this case, if $d(v_2)\geq 3$, then all neighbors of $v_2$ outside
$P$ are support vertices, each of which has exactly one
leaf-neighbor. Clearly, the component of $T-v_2v_3$ containing the
vertex $v_2$ is a subdivided star with at least two leaves. Let $D'$
be the set which is obtained from $D$ by deleting all support
vertices of the subdivided star. Next, one of the two cases as
following holds: (1) Each vertex of $D'$ is at distance at least $3$
from $v_1$; (2) There is a vertex of $D'$ which is within distance
$2$ of $v_1$. In both cases, the same arguments as Case~1 shows that
$T\in \mathscr{O}$.

We may assume that $d(v_2)=2$ by means of the above discussion.
Without loss of generality, $v_2\not \in D$ (Otherwise, replacing
$v_2$ in $D$ with $v_3$, and the resulting set is also a
$\gamma_{t2}$-set of $T$), and then $v_3\in D$. We may assume that
$|D|\geq 3$, for otherwise, we are done.

If there exists a vertex of $D\setminus \{v_1, v_3\}$ is within
distance $2$ of $v_3$. Analogously to Case~1, $T$ is obtained from
$T'$ by operation $\mathscr{O}_2$, where $T'=T-\{v_0, v_1\}$, and
$T\in \mathscr{O}$.

Thus, each vertex of $D\setminus \{v_1, v_3\}$ is at distance at
least $3$ from $v_3$. From the choice of $P$ and $D$, $v_3$ has only
one neighbor outside $P$ which is a leaf or $d(v_3)=2$.

In the former case, we consider $T'=T-v_0$ and it is easy to show
that $T\in \mathscr{O}$. In the latter case, suppose that
$d(v_4)\geq 3$ and let $v_3'$ be a neighbor of $v_4$ outside $P$.
From the choice of $v_2$ , $v_3$ and $D$, the component of
$T-v_3'v_4$ containing $v_3'$ is either a subdivided star or a
$P_4$. We only need to consider the second case. Let $v_2'$ be the
neighbor of $v_3'$ on the $P_4$, and $v_1'$ be the remaining
neighbor of $v_2'$ on the $P_4$. Clearly, $v_1'\in D$. Since each
vertex of $D\setminus \{v_1, v_3\}$ is at distance at least $3$ from
$v_3$, $v_2'\in D$. Replacing $v_2'$ in $D$ with $v_3'$, and the
resulting set is also a $\gamma_{t2}$-set of $T$. Take $T'=T-\{v_0,
v_1\}$, and it can be deduced that $T'\in \mathscr{O}$ and $T$ is
obtained from $T'$ by operation $\mathscr{O}_2$.

Hence, we may assume that $d(v_4)=2$. We know that $v_1, v_3\in D$.
Because each vertex of $D\setminus \{v_1, v_3\}$ is at distance at
least $3$ from $v_3$. Then, let $T'=T-\{v_0, v_1, v_2, v_3, v_4\}$.
Observe that $D\setminus \{v_1, v_3\}$ is a semitotal dominating set
of $T'$. Moreover, we have that $\gamma(T')\leq \gamma_{t2}(T')\leq
\gamma_{t2}(T)-2=\gamma(T)-2\leq \gamma(T')$. Thus, $\gamma(T')=
\gamma_{t2}(T')$. By induction, $T'\in \mathscr{O}$. In addition,
let $D'$ be a $\gamma(T'; v_5)$-set of $T'$ and $D''=D'\cup \{v_1,
v_4\}$. We can see that $D''$ dominates all vertices of $T$. That
is, $\gamma(T'; v_5)+2\geq \gamma(T)$. It follows from
$\gamma(T)=\gamma(T')+2\geq \gamma(T'; v_5)+2\geq \gamma(T)$ that
$\gamma(T'; v_5)=\gamma(T')$. Hence, $T$ is obtained from $T'$ by
operation $\mathscr{O}_2$.

The proof is completed. \end{proof}

As an immediate consequence of Lemmas~3.1 and 3.2 we have the
following characterization of $(\gamma, \gamma_{t2})$-trees.

\begin{thm}
A tree $T$ is a $(\gamma, \gamma_{t2})$-tree if and only if $T\in
\mathscr{O}$.
\end{thm}

\acknowledgements We thank two reviewers for giving helpful
suggestions and comments to improve the manuscript.


\end{document}